\documentclass[12pt,bezier]{article}
\usepackage{times}
\usepackage{booktabs}
\usepackage{pifont}
\usepackage{floatrow}
\usepackage{caption}
\usepackage{mathrsfs}
\usepackage[fleqn]{amsmath}
\usepackage{amsfonts,amsthm,amssymb,mathrsfs,bbding}
\usepackage{txfonts}
\usepackage{graphics,multicol}
\usepackage{graphicx}
\usepackage{color}

\usepackage{caption}
\usepackage{cite}
\usepackage{latexsym,bm}
\usepackage{indentfirst}
\usepackage{color}
\usepackage[colorlinks=true,anchorcolor=blue,filecolor=blue,linkcolor=blue,urlcolor=blue,citecolor=blue]{hyperref}
\usepackage{extarrows}
\usepackage{cite}
\usepackage{latexsym,bm}
\usepackage{mathtools}
\evensidemargin=\oddsidemargin\topmargin=-1.5cm

\newtheorem{theorem}{Theorem}[section]
\newtheorem{problem}{Problem}[section]

\newtheorem{lemma}[theorem]{Lemma}

\newtheorem{conj}[theorem]{Conjecture}
\newtheorem{claim}{Claim}

\theoremstyle{definition}

\addtocounter{section}{0}

\begin{document}
\title{Spectral Tur\'{a}n problem of non-bipartite graphs: Forbidden books}
\author{{\bf Ruifang Liu},~{\bf Lu Miao}\thanks{Corresponding author.
E-mail addresses: rfliu@zzu.edu.cn (R. Liu); miaolu0208@163.com (L. Miao).}\\
{\footnotesize School of Mathematics and Statistics, Zhengzhou University, Zhengzhou, Henan 450001, China}}

\date{}
\maketitle
{\flushleft\large\bf Abstract}
A book graph $B_{r+1}$ is a set of $r+1$ triangles with a common edge, where $r\geq0$ is an integer. Zhai and Lin [J. Graph Theory 102 (2023) 502-520] proved that for  $n\geq\frac{13}{2}r$, if $G$ is a $B_{r+1}$-free graph of order $n$, then $\rho(G)\leq\rho(T_{n,2})$, with equality if and only if $G\cong T_{n,2}$. Note that the extremal graph $T_{n,2}$ is bipartite. Motivated by the above elegant result, we investigate the spectral Tur\'{a}n problem of non-bipartite $B_{r+1}$-free graphs of order $n$.
For $r=0$, Lin, Ning and Wu [Comb. Probab. Comput. 30 (2021) 258-270] provided a complete solution and proved a nice result: If $G$ is a non-bipartite triangle-free graph of order $n$, then $\rho(G)\leq\rho\big(SK_{\lfloor\frac{n-1}{2}\rfloor,\lceil\frac{n-1}{2}\rceil}\big)$, with equality if and only if $G\cong SK_{\lfloor\frac{n-1}{2}\rfloor,\lceil\frac{n-1}{2}\rceil}$, where $SK_{\lfloor\frac{n-1}{2}\rfloor,\lceil\frac{n-1}{2}\rceil}$ is the graph obtained from $K_{\lfloor\frac{n-1}{2}\rfloor,\lceil\frac{n-1}{2}\rceil}$ by subdividing an edge.

For general $r\geq1$, let $K_{\lfloor\frac{n-1}{2}\rfloor,\lceil\frac{n-1}{2}\rceil}^{r, r}$ be the graph obtained from $K_{\lceil\frac{n-1}{2}\rceil,\lfloor\frac{n-1}{2}\rfloor}$ by adding a new vertex $v_{0}$ such that $v_{0}$ has exactly $r$ neighbours in each part of $K_{\lceil\frac{n-1}{2}\rceil,\lfloor\frac{n-1}{2}\rfloor}$. By adopting a different technique named the residual index, Chv\'{a}tal-Hanson theorem and typical spectral extremal methods, we in this paper prove that: If $G$ is a non-bipartite $B_{r+1}$-free graph of order $n$, then $\rho(G)\leq\rho\Big(K_{\lfloor\frac{n-1}{2}\rfloor,\lceil\frac{n-1}{2}\rceil}^{r, r}\Big)$ , with equality if and only if $G\cong K_{\lfloor\frac{n-1}{2}\rfloor,\lceil\frac{n-1}{2}\rceil}^{r, r}$. An interesting phenomenon is that the spectral extremal graphs are completely different for $r=0$ and general $r\geq1$.

\begin{flushleft}
\end{flushleft}

\section{Introduction}\label{se1}
Let $G$ be a simple graph with vertex set $V(G)$ and edge set $E(G)$. For each vertex $u\in V(G)$, let $N(u)$ be the set of neighbours of $u$ in $V(G)$. The degree of $u$, denoted by $d(u)$, is the number of vertices in $N(u)$. For a subset $S$ of $V(G)$, denote by $N_{S}(u)$ and $d_{S}(u)$ the set and the number of neighbours of $u$ in $S$, respectively. Let $G[S]$ be the subgraph of $G$ induced by $S$, and let $|S|$ be the number of vertices in $S.$ For two disjoint subsets $A$ and $B$ of $V(G)$, we use $e(A,B)$ to denote the number of edges with one endpoint in $A$ and the other in $B$.

Given a graph $H$, a graph is said to be $H$-free if it does not contain a subgraph isomorphic to $H$. A classic problem in extremal graph theory is the Tur\'{a}n-type problem, which asks what is the maximum number of edges in an $H$-free graph of order $n$. The study of the Tur\'{a}n-type problem can date back to Mantel's theorem \cite{Mantel1907} in 1907, which states that a triangle-free graph $G$ of order $n$ has at most $\big\lfloor \frac{n^{2}}{4}\big\rfloor$ edges. In 1941, Tur\'{a}n \cite{Turan1941} extended Mantel's theorem and showed that $e(G)\leq e(T_{n,r})$ for every $K_{r+1}$-free graph $G$ of order $n\geq r+1\geq3$, where $T_{n,r}$ is the $r$-partite Tur\'{a}n graph. Since then, much attention has been paid to the Tur\'{a}n-type problem (see a survey, \cite{FS2013}).

Let $A(G)$ be the adjacency matrix of $G$. The largest eigenvalue of $A(G)$, denoted by $\rho(G)$, is called the spectral radius of $G$. By Perron-Frobenius theorem, every connected graph $G$ has a positive unit eigenvector corresponding to $\rho(G)$, which is called the \emph{Perron vector} of $G$. In 2007, Nikiforov \cite{Nikiforov2007} proved that $\rho(G)\leq \rho(T_{n,r})$ for every $n$-vertex $K_{r+1}$-free graph $G$, with equality if and only if $G\cong T_{n,r}$, which is known as the spectral Tur\'{a}n theorem. Furthermore, Nikiforov \cite{Nikiforov2} formally proposed a spectral version of the Tur\'{a}n-type problem as follows.

\begin{problem}[\!\cite{Nikiforov2}]\label{p1}
What is the maximum spectral radius of an $H$-free graph of order $n$?
\end{problem}

Focusing on Problem \ref{p1}, there have been a wealth of results in the past decades.
For example, see \cite{Nikiforov2007,Wilf,LP2023} for $K_{r+1}$-free graphs, \cite{CDM,Gao2019,LZZ2022,Nikiforov2008,Zhai2020} for spectral problems on cycles, \cite{Chen2019} for linear forests, \cite{CS2022} for odd wheels, \cite{CDM1} for a spectral Erd\H{o}s-S\'{o}s theorem, \cite{Tait2017,Lin2021} for planar graphs and outplanar graphs, \cite{Chen2024,HL2024,Tait2019,Zhai2022,ZFL} for $F$-minor-free graphs $(F=K_r-E(\bigcup_{i=1}^{l}P_{k_i}))$, $F_k$-minor-free graphs, $K_r$-minor-free graphs, $K_{s,t}$-minor-free graphs and general $H$-minor free graphs, \cite{CFTZ,DKL,Li2022} for intersecting graphs, \cite{Wang2023} for a spectral extremal conjecture, \cite{FTZ,Lei2024,NWK} for disjoint graphs, \cite{Zhai2023} for books and theta graphs, \cite{LN2023} for sparse spanning graphs, and \cite{LLF2022} for a comprehensive survey.

In particular, the spectral Tur\'{a}n problem on non-bipartite graphs of order $n$ has also attracted much attention. Let $SK_{\lfloor\frac{n-1}{2}\rfloor,\lceil\frac{n-1}{2}\rceil}$ be the graph obtained from $K_{\lfloor\frac{n-1}{2}\rfloor,\lceil\frac{n-1}{2}\rceil}$ by subdividing an edge. Lin, Ning and Wu \cite{LNW2021} proved that $\rho(G)\leq\rho\big(SK_{\lfloor\frac{n-1}{2}\rfloor,\lceil\frac{n-1}{2}\rceil}\big)$ for every non-bipartite triangle-free graph $G$ of order $n$, with equality if and only if $G\cong SK_{\lfloor\frac{n-1}{2}\rfloor,\lceil\frac{n-1}{2}\rceil}$. For this problem, Li and Peng \cite{LP2023} provided an alternative proof. Subsequently, Guo, Lin and Zhao \cite{Guo2021} showed that if $G$ is a non-bipartite $C_5$-free graph with order $n\geq21$, then $\rho(G)\leq\rho\big(K_{\lfloor\frac{n-2}{2}\rfloor,\lceil\frac{n-2}{2}\rceil}\bullet K_{3}\big)$, with equality if and only if $G\cong K_{\lfloor\frac{n-2}{2}\rfloor,\lceil\frac{n-2}{2}\rceil}\bullet K_{3}$, where $K_{\lfloor\frac{n-2}{2}\rfloor,\lceil\frac{n-2}{2}\rceil}\bullet K_{3}$ is the graph obtained by identifying a vertex of $K_{a,b}$ belonging to the part of size $b$ and a vertex of $K_3$. In fact, Zhang and Zhao \cite{ZZ2023} proved that $K_{\lfloor\frac{n-2}{2}\rfloor,\lceil\frac{n-2}{2}\rceil}\bullet K_{3}$ is still the unique spectral extremal graph among all non-bipartite $C_{2k+1}$-free graphs for $k\geq2$.
Furthermore, the maximum spectral radius of non-bipartite $\{C_3,C_5,\ldots,C_{2k+1}\}$-free graphs of order $n$ was determined independently by Lin and Guo \cite{LG2021} and Li, Sun and Yu \cite{LS2022}.

A graph is color-critical if it contains an edge whose removal reduces its chromatic number. In 2009, Nikiforov \cite{Nikiforov2009} proved that if $H$ is color-critical with $\chi(H)=k+1$, then there exists an $n_{0}(H)\geq e^{|V(H)|k^{(2k+9)(k+1)}}$ such that $\rho(G)\leq\rho(T_{n,k})$ for every $H$-free graph $G$ of order $n$, with equality if and only if $G\cong T_{n,k}$. Books, as a kind of color-critical graph, play an important role in the history of extremal graph theory. In 1962, Erd\H{o}s \cite{E} initiated the study of books in graphs. Since then books have attracted considerable attention both in extremal graph theory (see, e.g., \cite{EFG,EFR,KN}) and in Ramsey graph theory (see, e.g., \cite{NR2005,RS}).
A book graph $B_{r+1}$ is a set of $r+1$ triangles with a common edge, where $r\geq0$. It is easy to see that $B_{r+1}$ is color-critical with $\chi(B_{r+1})=3$. It follows from Nikiforov's result that $T_{n,2}$ attains the maximum spectral radius among all $B_{r+1}$-free graphs of order $n$. Note that $n_{0}(H)$ is exponential with $|V(H)|$. Very recently, Zhai and Lin \cite{Zhai2023} improved $n_{0}(H)$ to be a linear function on $|V(H)|$ for $H$ being a book graph.

\begin{theorem}[\!\cite{Zhai2023}]\label{th1}
Let $n$ and $r$ be integers with $n\geq\frac{13}{2}r$. If $G$ is a $B_{r+1}$-free graph of order $n$, then
\begin{eqnarray*}
\rho(G)\leq\rho(T_{n,2}),
\end{eqnarray*}
with equality if and only if $G\cong T_{n,2}$.
\end{theorem}

Note that the above extremal graph $T_{n,2}$ is bipartite. Hence a natural and interesting problem occurs:

\begin{problem}\label{p2}
What is the maximum spectral radius of non-bipartite $B_{r+1}$-free graphs of order $n$?
\end{problem}

Note that $B_{r+1}$ is isomorphic to a triangle for $r=0$. A nice result of Lin, Ning and Wu in \cite{LNW2021} indicated the unique spectral extremal graph attaining the maximum spectral radius among all non-bipartite triangle-free graphs of order $n$.

\begin{theorem}[\!\cite{LNW2021}]\label{th2}
Let $G$ be a non-bipartite triangle-free graph of order $n$. Then
\begin{eqnarray*}
\rho(G)\leq\rho\Big(SK_{\lfloor\frac{n-1}{2}\rfloor,\lceil\frac{n-1}{2}\rceil}\Big),
\end{eqnarray*}
with equality if and only if $G\cong SK_{\lfloor\frac{n-1}{2}\rfloor,\lceil\frac{n-1}{2}\rceil}$.
\end{theorem}

Concerning Problem \ref{p2}, we determine the maximum spectral radius among all non-bipartite $B_{r+1}$-free graphs of order $n$ and characterize the unique spectral extremal graph, where $r\geq1$. Let $K_{\lfloor\frac{n-1}{2}\rfloor,\lceil\frac{n-1}{2}\rceil}^{r, r}$ be a graph obtained from $K_{\lfloor\frac{n-1}{2}\rfloor,\lceil\frac{n-1}{2}\rceil}$ by adding a new vertex $v_{0}$ such that $v_{0}$ has $r$ neighbours in each part of $K_{\lfloor\frac{n-1}{2}\rfloor,\lceil\frac{n-1}{2}\rceil}$ (see Fig. \ref{f1}).

\begin{figure}\label{f1}
\centering
\includegraphics[width=0.65\textwidth]{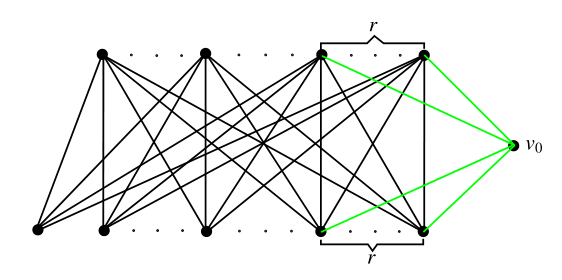}
\caption{The graph $K_{\lfloor\frac{n-1}{2}\rfloor,\lceil\frac{n-1}{2}\rceil}^{r, r}$.}\label{f1}
\end{figure}

\begin{theorem}\label{main}
Let $r\geq1$ and $n\geq8(r^{2}+r+4)$ be integers. If $G$ is a non-bipartite $B_{r+1}$-free graph of order $n$, then
\begin{eqnarray*}
\rho(G)\leq\rho\Big(K_{\lfloor\frac{n-1}{2}\rfloor,\lceil\frac{n-1}{2}\rceil}^{r, r}\Big),
\end{eqnarray*}
with equality if and only if $G\cong K_{\lfloor\frac{n-1}{2}\rfloor,\lceil\frac{n-1}{2}\rceil}^{r, r}$.
\end{theorem}

The rest of the paper is organized as follows. In Section 2, we first give some useful lemmas which play an important role in our proof. In Section 3, we present the proof of Theorem \ref{main}.

\section{Preliminaries}\label{se2}
In this section, we introduce some notation and important results, which are essential to the proof of our main theorem.

Let $G$ be a simple graph of order $n$ with matching number $\nu(G)$ and maximum degree $\Delta(G)$.
Given two positive integers $\nu$ and $\Delta$, we denote $$f(\nu, \Delta) = \max\{e(G): \nu(G) \leq \nu, \Delta(G) \leq \Delta\}.$$
Chv\'{a}tal and Hanson \cite{CH1976} showed that $f(\nu, \Delta)=\big\lfloor\frac{n\Delta}{2}\big\rfloor$ for $n\leq2\nu$. For $n\geq2\nu+1$, they proved the following result.

\begin{lemma}[\!\cite{CH1976}]\label{lem1}

Let $n$, $\nu$, $\Delta$ be positive integers with $n\geq2\nu+1$.\vspace{0.3cm}

\noindent (i) If $\Delta\leq2\nu$ and $n\leq2\nu+\Big\lfloor\frac{\nu}{\lfloor(\Delta+1)/2\rfloor}\Big\rfloor$, then
\begin{align*}
\begin{split}
f(\nu, \Delta)=\left\{
\begin{array}{ll}
\min\bigg\{\Big\lfloor\frac{n\Delta}{2}\Big\rfloor, \nu\Delta+\frac{\Delta-1}{2}\Big\lfloor\frac{2(n-\nu)}{\Delta+3}\Big\rfloor \bigg\}, & if~\Delta~is~odd, \\
\frac{n\Delta}{2}, & if~\Delta~is~even.
\end{array}
\right.
\end{split}
\end{align*}

\noindent (ii) If $\Delta\leq2\nu$ and $n\geq2\nu+\Big\lfloor\frac{\nu}{\lfloor(\Delta+1)/2\rfloor}\Big\rfloor+1$, then

$$f(\nu, \Delta)=\nu\Delta+\bigg\lfloor\frac{\nu}{\lfloor(\Delta+1)/2\rfloor}\bigg\rfloor\bigg\lfloor\frac{\Delta}{2}\bigg\rfloor.$$

\noindent (iii) If $\Delta\geq2\nu+1$, then
\begin{align*}
\begin{split}
f(\nu, \Delta)=\left\{
\begin{array}{ll}
\max\bigg\{\binom{2\nu+1}{2}, \Big\lfloor\frac{\nu(n+\Delta-\nu)}{2}\Big\rfloor\bigg\}, & if~n\leq\nu+\Delta, \\
\nu\Delta, & if~n\geq\nu+\Delta+1.\\
\end{array}
\right.
\end{split}
\end{align*}
\end{lemma}

Based on Lemma \ref{lem1} and simple calculations, one can easily check that for every two positive integers $\nu \geq 1$ and $\Delta\geq 1$,
\begin{eqnarray}\label{e0}
f(\nu, \Delta)\leq \nu(\Delta+1).
\end{eqnarray}

\begin{lemma}[\!\cite{Wu2005}]\label{lem2}
Let $G$ be a connected graph with $v_{i}, v_{j}\in V(G)$ and $S\subseteq N_{G}(v_{j})\backslash N_{G}(v_{i})$. Assume that $G'=G-\{v_{j}v|v\in S\}+\{v_{i}v|v\in S\}$ and $(x_{1},\ldots,x_{n})^{T}$ is the Perron vector of $G$, where $x_{i}$ corresponds to $v_{i}$. If $x_{i}\geq x_{j}$, then $\rho(G')>\rho(G)$.
\end{lemma}

Let $M$ be a real $n\times n$ matrix, and let $V(G)=\{1,2,\ldots,n\}$. Given a partition
$\Pi: V(G)=V_{1}\cup V_{2}\cup\cdots\cup V_{k}$, $M$ can be correspondingly partitioned as
$$
M=
\begin{pmatrix}
M_{1,1}&M_{1,2}&\cdots&M_{1,k}\\
M_{2,1}&M_{2,2}&\cdots&M_{2,k}\\
\vdots&\vdots&\ddots&\vdots\\
M_{k,1}&M_{k,2}&\cdots&M_{k,k}
\end{pmatrix}
.$$
The quotient matrix of $M$ with respect to $\Pi$ is defined as the $k\times k$ matrix $B_{\Pi} = (b_{i,j})^{k}_{i,j=1}$,
where $b_{i,j}$ is the average value of all row sums of $M_{i,j}$.
The partition $\Pi$ is called equitable if each block $M_{i,j}$ of $M$ has constant row sum $b_{i,j}$.
Also, we say that the quotient matrix $B_{\Pi}$ is equitable if $\Pi$ is an equitable partition of $M$.

\begin{lemma}[\!\cite{Brouwer2011,Godsil2001}]\label{lem3}
Let $M$ be a real symmetric matrix and $\lambda(M)$ be its largest eigenvalue.
Let $R_{\Pi}$ be an equitable quotient matrix of $M$. Then the eigenvalues of $R_{\Pi}$ are also eigenvalues of $M$.
Furthermore, if $M$ is nonnegative and irreducible, then $\lambda(M)=\lambda(R_{\Pi})$.
\end{lemma}

First we introduce the definition of the graph $K_{s,t}^{r_{1}, r_{2}}$. Let $S$ and $T$ be two parts of $K_{s,t}$ with $|S|=s$, $|T|=t$ and $s+t=n-1$, and let $r_{1}, r_{2}$ and $r$ be positive integers with $\max\{r_{1}, r_{2}\}\leq r\leq \min\{s,t\}$. Define $K_{s,t}^{r_{1}, r_{2}}$ to be a graph obtained from the complete bipartite graph $K_{s,t}$ by adding a new vertex $v_{0}$ such that $|N_{S}(v_{0})|=r_{1}$ and $|N_{T}(v_{0})|=r_{2}$ (see Fig. \ref{f2}).
\begin{figure}
\centering
\includegraphics[width=0.75\textwidth]{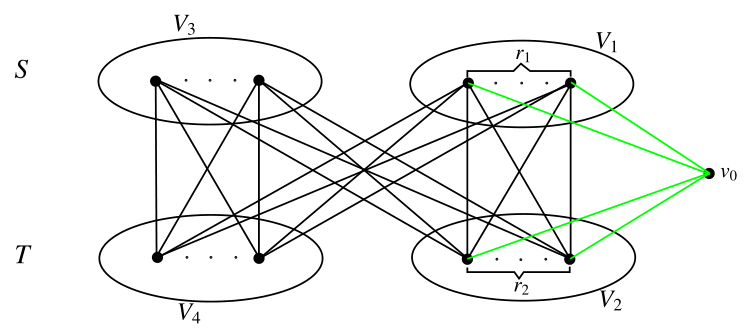}
\caption{The graph $K_{s,t}^{r_1, r_2}$.}
\label{f2}
\end{figure}
\begin{lemma}\label{lem4}
The spectral radius $\rho\big(K_{s,t}^{r, r}\big)$ of $K_{s,t}^{r, r}$ is the largest root of the equation
\begin{eqnarray*}
f(s,t,r,x)=x^{5}-(2r+st)x^{3}-2r^{2}x^{2}+(2str-sr^{2}-tr^{2})x=0.
\end{eqnarray*}
\end{lemma}
\begin{proof}
Let $V_1=N_{S}(v_0),$ $V_2=N_{T}(v_0),$ $V_3=S\backslash V_1$ and $V_4=T\backslash V_2$ (see Fig. \ref{f2}). The quotient matrix $R_{\Pi}$ of $K_{s,t}^{r, r}$ with respect to the equitable partition $\Pi: V(K_{s,t}^{r, r})=\{v_0\}\cup V_1\cup V_2\cup V_3\cup V_4$ is
$$
R_{\Pi}=
\begin{pmatrix}
0&r&r&0&0\\
1&0&r&0&t-r\\
1&r&0&s-r&0\\
0&0&r&0&t-r\\
0&r&0&s-r&0
\end{pmatrix}
.$$
Then the characteristic polynomial of $R_{\Pi}$ is
\begin{eqnarray*}
f(s,t,r,x)=x^{5}-(2r+st)x^{3}-2r^{2}x^{2}+(2str-sr^{2}-tr^{2})x.
\end{eqnarray*}
By Lemma \ref{lem3}, we have $\rho(K_{s,t}^{r, r})$ is the largest root of the equation $f(s,t,r,x)=0$.
\end{proof}

Now we are ready to estimate a better lower bound of $\rho\Big(K_{\lfloor\frac{n-1}{2}\rfloor,\lceil\frac{n-1}{2}\rceil}^{r, r}\Big)$ and $\rho^{2}\Big(K_{\lfloor\frac{n-1}{2}\rfloor,\lceil\frac{n-1}{2}\rceil}^{r, r}\Big)$.
\begin{lemma}\label{lem5}
Let $r\geq1$ and $n\geq 8(r^{2}+r+4)$. Then we have
\begin{eqnarray*}
\rho\Big(K_{\lfloor\frac{n-1}{2}\rfloor,\lceil\frac{n-1}{2}\rceil}^{r, r}\Big)>\frac{n-1}{2}-\frac{1}{4(n-1)}>4(r^{2}+r+2)
~~~and~~~\rho^{2}\Big(K_{\lfloor\frac{n-1}{2}\rfloor,\lceil\frac{n-1}{2}\rceil}^{r, r}\Big)>\bigg\lfloor\frac{(n-1)^{2}}{4}\bigg\rfloor.
\end{eqnarray*}
\end{lemma}
\begin{proof}
Let $s=t=\frac{n-1}{2}$ if $n$ is odd. It follows from Lemma \ref{lem4} that
$$f\Big(\frac{n-1}{2},\frac{n-1}{2},r,x\Big)=x\bigg(x^{4}-\Big(\frac{n^{2}}{4}-\frac{n}{2}+2r+\frac{1}{4}\Big)x^{2}-2r^{2}x-(n-1)r^{2}+\frac{1}{2}(n-1)^{2}r\bigg).$$
By Lemma \ref{lem3}, we have $\rho\Big(K_{\frac{n-1}{2},\frac{n-1}{2}}^{r, r}\Big)$ is the largest root of $f\big(\frac{n-1}{2},\frac{n-1}{2},r,x\big)=0$.
Since $f\big(\frac{n-1}{2},\frac{n-1}{2},r,\frac{n-1}{2}\big)=-(n-1)^{2}r^{2}<0,$ we have $\rho\Big(K_{\frac{n-1}{2},\frac{n-1}{2}}^{r, r}\Big)>\frac{n-1}{2}$, and hence $\rho^{2}\Big(K_{\frac{n-1}{2},\frac{n-1}{2}}^{r, r}\Big)>\frac{(n-1)^{2}}{4}=\big\lfloor\frac{(n-1)^{2}}{4}\big\rfloor$.

Let $s=\frac{n-2}{2}$ and $t=\frac{n}{2}$ if $n$ is even. Then
$$f\Big(\frac{n-2}{2},\frac{n}{2},r,x\Big)=x\bigg(x^{4}-\Big(\frac{n^{2}}{4}-\frac{n}{2}+2r\Big)x^{2}-2r^{2}x+\frac{1}{2}n^{2}r-nr^{2}-nr+r^{2}\bigg).$$
It follows from Lemma \ref{lem3} that $\rho\Big(K_{\frac{n-2}{2},\frac{n}{2}}^{r, r}\Big)$ is the largest root of $f\big(\frac{n-2}{2},\frac{n}{2},r,x\big)=0$.
Since $f\big(\frac{n-2}{2},\frac{n}{2},r,\frac{n-1}{2}-\frac{1}{4(n-1)}\big)<0,$ we have $\rho\Big(K_{\frac{n-2}{2},\frac{n}{2}}^{r, r}\Big)>\frac{n-1}{2}-\frac{1}{4(n-1)}$. Furthermore, $\rho^{2}\Big(K_{\frac{n-2}{2},\frac{n}{2}}^{r, r}\Big)>\frac{(n-1)^{2}}{4}-\frac{1}{4}+\frac{1}{16(n-1)^{2}}>\big\lfloor\frac{(n-1)^{2}}{4}\big\rfloor$.

Therefore, regardless of the parity of $n$, we can always obtain that $\rho\Big(K_{\lfloor\frac{n-1}{2}\rfloor,\lceil\frac{n-1}{2}\rceil}^{r, r}\Big)>\frac{n-1}{2}-\frac{1}{4(n-1)}$ and $\rho^{2}\Big(K_{\lfloor\frac{n-1}{2}\rfloor,\lceil\frac{n-1}{2}\rceil}^{r, r}\Big)>\big\lfloor\frac{(n-1)^{2}}{4}\big\rfloor$. Note that $n\geq 8(r^{2}+r+4)$. Then we have $\rho\Big(K_{\lfloor\frac{n-1}{2}\rfloor,\lceil\frac{n-1}{2}\rceil}^{r, r}\Big)>\frac{n-1}{2}-\frac{1}{4(n-1)}>4(r^{2}+r+2)$.
\end{proof}

\section{Proofs}\label{se3}
Suppose that $G^{*}$ attains the maximum spectral radius among all the non-bipartite $B_{r+1}$-free graphs of order $n$, where $r\geq1$ and $n\geq8(r^{2}+r+4)$. Firstly, we claim that $G^{*}$ is connected. Otherwise, assume that $G_{1}$ is a connected component of $G^{*}$ with $\rho(G_{1})=\rho(G^{*})$. By adding an edge $e_{0}$ between $G_{1}$ and any other connected component of $G^{*},$ we get a new graph $G_{2}.$ Note that $e_0$ is a cut edge. Then $G_{2}$ is still a non-bipartite $B_{r+1}$-free graph. However, $\rho(G_{2})>\rho(G_{1})=\rho(G^{*}),$ which contradicts the maximality of $\rho(G^{*}).$ So $G^{*}$ is connected.  Let $\textbf{x}=(x_{v})_{v\in V(G^{*})}$ be the Perron vector of $G^{*}$. Let $u^{*}$ be a vertex of $G^{*}$ such that $x_{u^{*}}=\max\{x_{v}: v\in V(G^{*})\}$. Define $A=N(u^{*})$ and $B=V(G^{*})\backslash(\{u^{*}\}\cup A)$.
Note that $K_{\lfloor\frac{n-1}{2}\rfloor,\lceil\frac{n-1}{2}\rceil}^{r, r}$ is a non-bipartite $B_{r+1}$-free graph of order $n$.
By Lemma \ref{lem5}, we have
\begin{eqnarray}\label{eq1}
\rho(G^{*})\geq\rho\Big(K_{\lfloor\frac{n-1}{2}\rfloor,\lceil\frac{n-1}{2}\rceil}^{r, r}\Big)>\frac{n-1}{2}-\frac{1}{4(n-1)}>4(r^{2}+r+2)
\end{eqnarray}
and
\begin{eqnarray}\label{eq2}
\rho^{2}(G^{*})\geq\rho^{2}\Big(K_{\lfloor\frac{n-1}{2}\rfloor,\lceil\frac{n-1}{2}\rceil}^{r, r}\Big)>\bigg\lfloor\frac{(n-1)^{2}}{4}\bigg\rfloor.
\end{eqnarray}

In the following, we write $\rho(G^{*})$ as $\rho$ for short. First, we prove a lower bound of the number of vertices in $A$.

\begin{lemma}\label{cla1}
$|A|\geq\big\lceil\frac{n}{2}\big\rceil$.
\end{lemma}
\begin{proof}
Using the eigen-equation, we have $\rho x_{u^{*}}=\sum_{u\in A}x_{u}\leq|A|x_{u^{*}}$. From the proof of Lemma \ref{lem5}, we know that $|A|\geq\rho>\frac{n-1}{2}$ for odd $n$, and so $|A|\geq\frac{n+1}{2}$. If $n$ is even, then $|A|\geq\rho>\frac{n-1}{2}-\frac{1}{4(n-1)}$, which implies that $|A|\geq\frac{n}{2}$. Thus, we have $|A|\geq\big\lceil\frac{n}{2}\big\rceil$.
\end{proof}

By Lemma \ref{lem1}, the number of edges in $A$ can also be well estimated.

\begin{lemma}\label{cla2}
$d_{A}(u)\leq r$ for each vertex $u\in A$. Moreover, $e(A)\leq\nu(G^{*}[A])(r+1)$.
\end{lemma}
\begin{proof}
Suppose to the contrary that there exists a vertex $u'\in A$ with $d_{A}(u')\geq r+1$. Then $G^{*}[\{u',u^{*}\}\cup N_{A}(u')]$ contains a book $B_{r+1}$, which contradicts that $G^{*}$ is $B_{r+1}$-free. Hence $d_{A}(u)\leq r$ for each vertex $u\in A$, and so $\Delta(G^{*}[A])\leq r$. By (\ref{e0}), we have
\begin{eqnarray*}
e(A)\leq f\big(\nu(G^{*}[A]), r\big)\leq\nu(G^{*}[A])(r+1),
\end{eqnarray*}
as desired.
\end{proof}

Now we can derive a good lower bound of the number of vertices in $B$.

\begin{lemma}\label{cla3}
For $n\geq8(r^{2}+r+4)$ and $r\geq1$, we have $|B|+1\geq\frac{n}{4}\geq5r+7.$
\end{lemma}
\begin{proof}
By Lemma \ref{cla1}, we have $|A|\geq\big\lceil\frac{n}{2}\big\rceil$. Now we assume that $|A|=\frac{n}{2}+\theta,$ where $\theta\geq0$. Note that $|A|+|B|+1=n$. Then $|B|+1=\frac{n}{2}-\theta$. By using the eigen-equation and Lemma \ref{cla2}, we have
\begin{eqnarray*}
\rho^{2}x_{u^{*}}&=&|A|x_{u^{*}}+\sum_{u\in A}d_{A}(u)x_{u}+\sum_{w\in B}d_{A}(w)x_{w}\\
&\leq&|A|x_{u^{*}}+r|A|x_{u^{*}}+|A||B|x_{u^{*}}\\
&=&|A|\big(|B|+1\big)x_{u^{*}}+r|A|x_{u^{*}}\\
&=&\Bigg(\frac{n^{2}}{4}-\theta^{2}+r\bigg(\frac{n}{2}+\theta\bigg)\Bigg)x_{u^{*}}.
\end{eqnarray*}
By (\ref{eq2}), we know that $\rho^{2}>\big\lfloor\frac{(n-1)^{2}}{4}\big\rfloor\geq\frac{n^{2}-2n}{4}$. Hence $\theta^{2}-r\theta<\frac{(r+1)n}{2}$.
By solving the quadratic inequality $\theta^{2}-r\theta-\frac{(r+1)n}{2}<0$, we have $\theta<\frac{r}{2}+\frac{\sqrt{r^{2}+2(r+1)n}}{2}$. Since $n\geq8(r^{2}+r+4)$, $\theta<\frac{r}{2}+\frac{\sqrt{r^{2}+2(r+1)n}}{2}<\sqrt{(r+1)n}$.
Then $$|B|+1=\frac{n}{2}-\theta>\frac{n}{2}-\sqrt{(r+1)n}\geq\frac{n}{4}\geq5r+7.$$
The result follows.
\end{proof}

Let $E(A)$ and $E(B)$ be the set of edges with both endpoints in $A$ and in $B$, respectively.

\begin{lemma}\label{cla4}
For each edge $u_{1}u_{2}\in E(A)$, $d_{B}(u_{1})+d_{B}(u_{2})\leq|B|+r-1$; for each edge $w_{1}w_{2}\in E(B)$, $d_{A}(w_{1})+d_{A}(w_{2})\leq|A|+r$.
\end{lemma}
\begin{proof}
If there exists some edge $u_{1}u_{2}\in E(A)$ such that $d_{B}(u_{1})+d_{B}(u_{2})\geq|B|+r$, then
\begin{eqnarray*}
|B|+r&\leq&d_{B}(u_{1})+d_{B}(u_{2})=|N_{B}(u_{1})\cup N_{B}(u_{2})|+|N_{B}(u_{1})\cap N_{B}(u_{2})|\\
&\leq&|B|+|N_{B}(u_{1})\cap N_{B}(u_{2})|.
\end{eqnarray*}
Hence $|N_{B}(u_{1})\cap N_{B}(u_{2})|\geq r$. Note that $u^{*}\in N_{G^{*}}(u_{1})\cap N_{G^{*}}(u_{2})$. It follows that $|N_{G^{*}}(u_{1})\cap N_{G^{*}}(u_{2})|\geq r+1$, which implies that $G^{*}$ contains a book $B_{r+1}$, a contradiction. Similarly, we can prove that $d_{A}(w_{1})+d_{A}(w_{2})\leq|A|+r$ for each edge $w_{1}w_{2}\in E(B)$.
\end{proof}

\begin{lemma}\label{cla5}
For each edge $u_{1}u_{2}\in E(A)$, we have
\begin{eqnarray*}
x_{u_{1}}+x_{u_{2}}\leq\frac{|B|+r+1}{\rho-r}x_{u^{*}}.
\end{eqnarray*}
\end{lemma}
\begin{proof}
Let $u_{1}u_{2}$ be an edge in $E(A)$ such that $x_{u_1}+x_{u_2}=\max\{x_{u}+x_{v}: uv\in E(A)\}$. It suffices to prove that $x_{u_{1}}+x_{u_{2}}\leq\frac{|B|+r+1}{\rho-r}x_{u^{*}}$.
By Lemma \ref{cla2}, we have
\begin{eqnarray*}
\rho x_{u_{1}}=x_{u^{*}}+\sum_{u\in N_{A}(u_{1})}x_{u}+\sum_{w\in N_{B}(u_{1})}x_{w}\leq x_{u^{*}}+rx_{u_2}+d_{B}(u_{1})x_{u^{*}}
\end{eqnarray*}
and
\begin{eqnarray*}
\rho x_{u_{2}}=x_{u^{*}}+\sum_{u\in N_{A}(u_{2})}x_{u}+\sum_{w\in N_{B}(u_{2})}x_{w}\leq x_{u^{*}}+rx_{u_1}+d_{B}(u_{2})x_{u^{*}}.
\end{eqnarray*}
Summing the above two inequalities and combining Lemma \ref{cla4}, we can obtain that
\begin{eqnarray*}
(\rho-r)(x_{u_{1}}+x_{u_{2}})=\big(d_{B}(u_{1})+d_{B}(u_{2})+2\big)x_{u^{*}}\leq(|B|+r+1)x_{u^{*}}.
\end{eqnarray*}
By (\ref{eq1}), then $\rho>4(r^{2}+r+1)>r$, and so the result follows.
\end{proof}

We use $\overline{d}_{A}(u)$ and $\overline{d}_{B}(u)$ to denote the number of non-adjacent vertices of $u$ in $A$ and $B$, respectively.
Let $\overline{e}(A, B)$ be the number of non-edges between $A$ and $B$, that is, $\overline{e}(A, B)=|A||B|-e(A, B)$.
In the following, we estimate the number of non-edges between $A$ and $B$.

\begin{lemma}\label{cla6}
$\overline{e}(A, B)\geq\nu(G^{*}[A])\big(|B|-r+1\big)$.
\end{lemma}
\begin{proof}
By Lemma \ref{cla4}, we have
\begin{eqnarray}\label{eq4}
\overline{d}_{B}(u_1)+\overline{d}_{B}(u_2)=2|B|-\Big(d_{B}(u_{1})+d_{B}(u_{2})\Big)\geq2|B|-\big(|B|+r-1\big)=|B|-r+1
\end{eqnarray}
for each edge $u_1u_2\in E(A)$. Note that there are $\nu(G^{*}[A])$ independent edges in $E(A)$. Then $\overline{e}(A, B)\geq\nu(G^{*}[A])(|B|-r+1)$.
\end{proof}

By Lemmas \ref{cla2}-\ref{cla6}, we are ready to prove an upper bound of $\nu(G^{*}[A])$, which is critical for us to characterize the structure of $G^{*}$.

\begin{lemma}\label{cla7}
$\nu(G^{*}[A])\leq1$ for $n\geq8(r^{2}+r+4)$.
\end{lemma}
\begin{proof}
Suppose to the contrary that $\nu(G^{*}[A])\geq2$. By Lemmas \ref{cla2}, \ref{cla5} and \ref{cla6}, we have
\begin{eqnarray}\label{eq5}
\rho^{2}x_{u^{*}}&=&|A|x_{u^{*}}+\sum_{u\in A}d_{A}(u)x_{u}+\sum_{w\in B}d_{A}(w)x_{w}\nonumber\\
&\leq&|A|x_{u^{*}}+\sum_{u_1u_2\in E(A)}(x_{u_1}+x_{u_2})+e(A, B)x_{u^{*}}\nonumber\\
&\leq&|A|x_{u^{*}}+e(A)\frac{|B|+r+1}{\rho-r}x_{u^{*}}+\Big(|A||B|-\bar{e}(A, B)\Big)x_{u^{*}}\nonumber\\
&\leq&|A|(|B|+1)x_{u^{*}}+\nu(G^{*}[A])\Bigg(\frac{(r+1)\big(|B|+r+1\big)}{\rho-r}-\big(|B|-r+1\big)\Bigg)x_{u^{*}}.
\end{eqnarray}
It follows from (\ref{eq1}) that $\rho>4(r^{2}+r+2)>5r+4$, which implies that $\frac{r+1}{\rho-r}<\frac{1}{4}$.
By Lemma \ref{cla3}, we know that $|B|+1\geq5r+7$, and hence $\frac{5}{4}r<\frac{1}{4}(|B|+1)$.
Combining (\ref{eq5}), we have
\begin{eqnarray*}
\rho^{2}&<&|A|(|B|+1)+\nu(G^{*}[A])\Bigg(\frac{|B|+r+1}{4}-\big(|B|-r+1\big)\Bigg)\\
&=&|A|(|B|+1)+\nu(G^{*}[A])\Bigg(-\frac{1}{2}\big(|B|+1\big)-\frac{1}{4}\big(|B|+1\big)+\frac{5}{4}r\Bigg)\\
&\leq&|A|(|B|+1)-\frac{\nu(G^{*}[A])}{2}\big(|B|+1\big)\\
&\leq&(|A|-1)(|B|+1)\\
&\leq&\bigg\lfloor\frac{(n-1)^{2}}{4}\bigg\rfloor,
\end{eqnarray*}
which contradicts (\ref{eq2}). The last inequality holds because $(|A|-1)+(|B|+1)=n-1$.
\end{proof}

Based on Lemma \ref{cla7}, we can further prove an upper bound of $\nu(G^{*}[B])$.

\begin{lemma}\label{cla8}
$\nu(G^{*}[B])\leq1$.
\end{lemma}
\begin{proof}
We prove the result by contradiction. Assume that $\nu(G^{*}[B])\geq2$.
By Lemma \ref{cla4}, we have
\begin{eqnarray}\label{eq6}
\overline{d}_{A}(w_1)+\overline{d}_{A}(w_2)\geq2|A|-(|A|+r)=|A|-r
\end{eqnarray}
for each edge $w_1w_2\in E(B)$. Note that there exist $\nu(G^{*}[B])$ independent edges in $G^{*}[B]$. Then we obtain that
\begin{eqnarray}\label{eq6.6}
\overline{e}(A, B)\geq\nu(G^{*}[B])(|A|-r).
\end{eqnarray}
Moreover, by Lemma \ref{cla7}, we have $\nu(G^{*}[A])=0$ or $\nu(G^{*}[A])=1$.

\vspace{1.8mm}
\noindent{\bf Case 1.} $G^{*}[A]$ is an empty graph or a star $K_{1,a}$ with possibly some isolated vertices.
\vspace{1.8mm}

At this moment, $e(A)=0$ or $e(A)=a$. By Lemma 3.2, we have $a\leq r$.
Now, by the eigen-equation again and (\ref{eq6.6}), we can obtain that
\begin{eqnarray*}
\rho^{2}x_{u^{*}}&=&|A|x_{u^{*}}+\sum_{u\in A}d_{A}(u)x_{u}+\sum_{w\in B}d_{A}(w)x_{w}\\
&\leq&|A|x_{u^{*}}+2e(A)x_{u^{*}}+e(A, B)x_{u^{*}}\\
&\leq&|A|x_{u^{*}}+2rx_{u^{*}}+\Big(|A||B|-\nu(G^{*}[B])\big(|A|-r\big)\Big)x_{u^{*}}\\
&\leq&|A|x_{u^{*}}+2rx_{u^{*}}+\Big(|A||B|-2\big(|A|-r\big)\Big)x_{u^{*}}\\
&=&\big(|A||B|-|A|+4r\big)x_{u^{*}}.
\end{eqnarray*}
Note that $|A|+|B|=n-1$. Then $|A||B|\leq\big\lfloor\frac{(n-1)^{2}}{4}\big\rfloor$. Moreover, by Lemma \ref{cla1}, we have $|A|\geq\big\lceil\frac{n}{2}\big\rceil\geq4(r^{2}+r+4)$, which implies that $|A|-4r>0$. Therefore, we obtain that $\rho^{2}\leq\big\lfloor\frac{(n-1)^{2}}{4}\big\rfloor-(|A|-4r)<\big\lfloor\frac{(n-1)^{2}}{4}\big\rfloor$, contradicting (\ref{eq2}).

\vspace{1.8mm}
\noindent{\bf Case 2.} $G^{*}[A]$ is a triangle with possibly some isolated vertices.
\vspace{1.8mm}

Then $e(A)=3$, and hence we have
\begin{eqnarray*}
\rho^{2}x_{u^{*}}&\leq&|A|x_{u^{*}}+2e(A)x_{u^{*}}+e(A, B)x_{u^{*}}\\
&\leq&|A|x_{u^{*}}+6x_{u^{*}}+\Big(|A||B|-2\big(|A|-r\big)\Big)x_{u^{*}}\\
&\leq&\Big(\Big\lfloor\frac{(n-1)^{2}}{4}\Big\rfloor-(|A|-2r-6)\Big)x_{u^{*}}.
\end{eqnarray*}
Note that $|A|\geq4(r^{2}+r+4)$. Then $|A|-2r-6>0$, and hence $\rho^{2}<\big\lfloor\frac{(n-1)^{2}}{4}\big\rfloor$, which is a contradiction.
\end{proof}

By Lemmas \ref{cla7} and \ref{cla8}, we are in a position to determine the relation between $\nu(G^{*}[A])$ and $\nu(G^{*}[B])$.

\begin{lemma}\label{cla9}
$\nu(G^{*}[A])+\nu(G^{*}[B])=1$.
\end{lemma}
\begin{proof}
By Lemma \ref{cla7}, we know that $\nu(G^{*}[A])=0$ or $\nu(G^{*}[A])=1$. By Lemma \ref{cla8}, we have $\nu(G^{*}[B])=0$ or $\nu(G^{*}[B])=1$.
If $\nu(G^{*}[A])=\nu(G^{*}[B])=0$, then $e(A)=e(B)=0$, and hence $G^{*}$ is a bipartite graph, a contradiction. If $\nu(G^{*}[A])=\nu(G^{*}[B])=1$, then there exist at least an edge $u_1u_2\in E(A)$ and at least an edge $w_1w_2\in E(B)$. Next we divide the proof into the following two cases according to the structure of $G^{*}[A]$.

\vspace{1.8mm}
\noindent{\bf Case 1.} $G^{*}[A]$ is a star $K_{1,a}$ with possibly some isolated vertices.
\vspace{1.8mm}

By Lemma \ref{cla2}, $e(A)=a\leq r$. Combining (\ref{eq4}) and (\ref{eq6}), we have $$\overline{d}_{B}(u_1)+\overline{d}_{B}(u_2)\geq|B|-r+1$$ and
$$\overline{d}_{A}(w_1)+\overline{d}_{A}(w_2)\geq|A|-r.$$
Note that there are at most four non-edges between $u_1u_2$ and $w_1w_2$. Then these non-edges are counted repeatedly at most four times in $\overline{d}_{B}(u_1)+\overline{d}_{B}(u_2)+\overline{d}_{A}(w_1)+\overline{d}_{A}(w_1)$. Hence we obtain that
\begin{eqnarray*}
\rho^{2}x_{u^{*}}&=&|A|x_{u^{*}}+\sum_{u\in A}d_{A}(u)x_{u}+\sum_{w\in B}d_{A}(w)x_{w}\\
&\leq&|A|x_{u^{*}}+2e(A)x_{u^{*}}+e(A, B)x_{u^{*}}\\
&\leq&|A|x_{u^{*}}+2rx_{u^{*}}+\Big(|A||B|-\big(|B|-r+1\big)-\big(|A|-r\big)+4\Big)x_{u^{*}}\\
&=&\bigg(|A||B|-\Big(|B|+1-4(r+1)\Big)\bigg)x_{u^{*}}.
\end{eqnarray*}
It follows from Lemma \ref{cla3} that $|B|+1\geq5r+7$, which implies that $|B|+1>4(r+1)$. Thus, we have $\rho^{2}\leq|A||B|\leq\big\lfloor\frac{(n-1)^{2}}{4}\big\rfloor$,
which contradicts (\ref{eq2}).

\vspace{1.8mm}
\noindent{\bf Case 2.} $G^{*}[A]$ is a triangle with possibly some isolated vertices.
\vspace{1.8mm}

In this case, we have
\begin{eqnarray*}
\rho^{2}x_{u^{*}}&\leq&|A|x_{u^{*}}+6x_{u^{*}}+\Big(|A||B|-\big(|B|-r+1\big)-\big(|A|-r\big)+4\Big)x_{u^{*}}\\
&=&\bigg(|A||B|-\Big(|B|+1-2(r+5)\Big)\bigg)x_{u^{*}}.
\end{eqnarray*}
By Lemma \ref{cla3}, we have $|B|+1\geq2(r+5)$, which implies that $\rho^{2}\leq|A||B|\leq\big\lfloor\frac{(n-1)^{2}}{4}\big\rfloor$, a contradiction.
\end{proof}

For any vertex $w\in B$, we define the residual index of $w$ $$\Gamma_{w}=d_{A}(w)(x_{u^{*}}-x_{w}),$$
which will be frequently used in the following proof.
Note that $x_{u^{*}}\geq x_{w}$. Then $\Gamma_{w}\geq0$. By using the eigen-equation again, we have
\begin{eqnarray}\label{eq7}
\rho^{2}x_{u^{*}}&=&|A|x_{u^{*}}+\sum_{u\in A}d_{A}(u)x_{u}+\sum_{w\in B}d_{A}(w)x_{w}\nonumber\\
&=&|A|x_{u^{*}}+\sum_{u\in A}d_{A}(u)x_{u}+\sum_{w\in B}d_{A}(w)x_{u^{*}}-\sum_{w\in B}\Gamma_{w}\nonumber\\
&\leq&|A|x_{u^{*}}+2e(A)x_{u^{*}}+\Big(|A||B|-\overline{e}(A, B)\Big)x_{u^{*}}-\sum_{w\in B}\Gamma_{w}\nonumber\\
&=&|A|x_{u^{*}}+2e(A)x_{u^{*}}+|A||B|x_{u^{*}}-\bigg(\sum_{w\in B}\overline{d}_{A}(w)x_{u^{*}}+\sum_{w\in B}\Gamma_{w}\bigg).
\end{eqnarray}

Now we are ready to show that $e(A)\neq0$. Firstly, we give the following lemma which plays an important role in the proof of Theorem \ref{cla11}.

\begin{lemma}\label{cla10}
If $e(A)=0$, then $\sum_{w\in B}\overline{d}_{A}(w)x_{u^{*}}+\sum_{w\in B}\Gamma_{w}<|A|x_{u^{*}}$.
\end{lemma}
\begin{proof}
Suppose to the contrary that $\sum_{w\in B}\overline{d}_{A}(w)x_{u^{*}}+\sum_{w\in B}\Gamma_{w}\geq|A|x_{u^{*}}$. Combining (\ref{eq7}) and $e(A)=0$, we have
$$\rho^{2}\leq|A||B|\leq\bigg\lfloor\frac{(n-1)^{2}}{4}\bigg\rfloor,$$
which contradicts (\ref{eq2}). The proof is completed.
\end{proof}

By using Lemma \ref{cla10} repeatedly, we next can prove that $e(A)\neq0$, which is very important to determine the structure of $G^{*}.$

\begin{theorem}\label{cla11}
$e(A)\neq0$.
\end{theorem}
\begin{proof}
Assume that $e(A)=0$. Then $\nu(G^{*}[A])=0$. By Lemma \ref{cla9}, $\nu(G^{*}[B])=1$.

\vspace{1.8mm}
\noindent{\bf Case 1.} $G^{*}[B]$ is a star $K_{1,b}$ with possibly some isolated vertices.
\vspace{1.8mm}

In this case, we have $e(B)=b$. Next we divide the proof into the following two cases according to different values of $b$.

\vspace{1.8mm}
\noindent{\bf Case 1.1.} $b=1$.
\vspace{1.8mm}

Let $w_0w_1$ be the only edge in $G^{*}[B]$. Since $G^{*}$ is non-bipartite, we have $N_{A}(w_0)\neq\varnothing$ and $N_{A}(w_1)\neq\varnothing$. Without loss of generality, assume that $x_{w_0}\leq x_{w_1}$. Now we prove the following claim.

\begin{claim}\label{claa}
$N_{A}(w_0)\subseteq N_{A}(w_1)$.
\end{claim}
\begin{proof}
Suppose to the contrary that $N_{A}(w_0)\nsubseteq N_{A}(w_1)$. Then there exists a vertex $u_0\in N_{A}(w_0)\backslash N_{A}(w_1)$. Let $G'=G^{*}-u_0w_0+u_0w_1$. Then $G'$ is still $B_{r+1}$-free. Otherwise, $G'$ contains a $B_{r+1}$ and $u_0w_1\in E(B_{r+1})$, and hence $u_0$ and $w_1$ have at least one common neighbour in $G'$, which contradicts that $N_{G'}(u_0)\cap N_{G'}(w_1)=\varnothing$.

Next we prove that $G'$ is still non-bipartite. If $G'$ is bipartite, then $u_0w_0\in E(C)$, where $C$ is an arbitrary odd cycle in $G^{*}$. Since $e(A)=0$, we have $w_0w_1\in E(C)$. Note that $u_0w_1\notin E(G^{*})$. Then $C$ must be a 5-cycle. Let $C_1,C_2,\ldots,C_t$ be all the 5-cycles in $G^{*}$. Denote by $u_i$ the neighbour of $w_1$ in $C_i\cap A$, where $1\leq i\leq t$ (see Fig. \ref{f3}).
Moreover, $u_iw_0, uw_0, uw_1\notin E(G^{*})$ for $1\leq i\leq t$ and each vertex $u\in A\backslash\{u_0,u_1\ldots,u_t\}$.
Otherwise, we can obtain new odd cycles of length at least $3$, a contradiction.
By the maximality of $\rho(G^{*})$, then $G^{*}$ is also edge-maximal, and hence $N_{G^{*}}(w)=A$ for each vertex $w\in B\backslash\{w_0,w_1\}$.
Now we determine the structure of $G^{*}$ as Fig. \ref{f3}.
Define $G''=G^{*}+w_0u_1$. Note that $r\geq1$. Then $G''$ is still a non-bipartite $B_{r+1}$-free graph and $\rho(G'')>\rho(G^{*})$, which contradicts the maximality of $\rho(G^{*})$. Hence $G'$ is non-bipartite.
\begin{figure}
\centering
\includegraphics[width=0.65\textwidth]{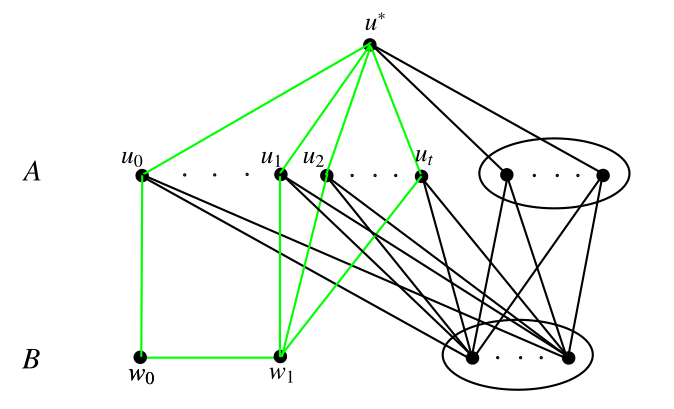}
\caption{The graph $G^{*}$ with $N_{A}(w_0)\nsubseteq N_{A}(w_1)$.}
\label{f3}
\end{figure}
Now we obtain that $G'$ is $B_{r+1}$-free and non-bipartite. By Lemma \ref{lem2}, then $\rho(G')>\rho(G^{*})$, which contradicts the maximality of $\rho(G^{*})$.
\end{proof}

\begin{claim}\label{clab}
$d_{A}(w_0)=r$.
\end{claim}
\begin{proof}
By Claim \ref{claa}, we have $N_{A}(w_0)\subseteq N_{A}(w_1)$. Since $G^{*}$ is $B_{r+1}$-free, $d_{A}(w_0)\leq r$. Furthermore, if $d_{A}(w_0)\leq r-1$, then we can construct a new graph $G'$ obtained from $G^{*}$ by adding $r-d_{A}(w_0)$ edges between $w_0$ and $N_{A}(w_1)\backslash N_{A}(w_0)$. It is obvious that $G'$ is still a non-bipartite $B_{r+1}$-free graph with $\rho(G')>\rho(G^{*})$, a contradiction.
Hence $d_{A}(w_0)=r$.
\end{proof}

Moreover, we can prove that
\begin{eqnarray}\label{eq8}
d_{A}(w_1)\leq|A|-1.
\end{eqnarray}
In fact, if $d_{A}(w_1)=|A|$, then $\rho x_{w_1}=x_{w_0}+\sum_{u\in A}x_{u}=x_{w_0}+\rho x_{u^{*}}>\rho x_{u^{*}}$, and hence $x_{w_1}>x_{u^{*}}$, which contradicts the choice of $u^{*}$.
By the maximality of $\rho(G^{*})$, we have $N_{A}(w)=A$ for each vertex $w\in B\backslash\{w_0, w_1\}$. Now we obtain the structure of $G^{*}$ as Fig. \ref{f4}.
\begin{figure}
\centering
\includegraphics[width=0.65\textwidth]{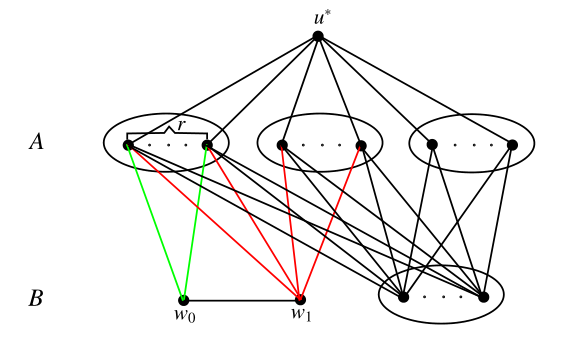}
\caption{The graph $G^{*}$ with $e(B)=1$.}
\label{f4}
\end{figure}
By Claim 2, we have $d_{A}(w_0)=r$, and hence $\overline{d}_{A}(w_0)=|A|-r$. By (\ref{eq8}), it follows that $\overline{d}_{A}(w_1)\geq1$. Now, let $N_{A}(w_0)=\{u_1,\ldots,u_r\}$. Since $\rho x_{w_0}=x_{w_1}+\sum_{i=1}^{r}x_{u_i}\leq(r+1)x_{u^{*}}$, $x_{w_0}\leq\frac{r+1}{\rho}x_{u^{*}}$. Hence $\Gamma_{w_0}=d_{A}(w_0)(x_{u^{*}}-x_{w_0})=r(x_{u^{*}}-x_{w_0})\geq r\big(1-\frac{r+1}{\rho}\big)x_{u^{*}}$. Then we have
\begin{eqnarray*}
\sum_{w\in B}\overline{d}_{A}(w)x_{u^{*}}+\sum_{w\in B}\Gamma_{w}&\geq&\overline{d}_{A}(w_0)x_{u^{*}}+\overline{d}_{A}(w_1)x_{u^{*}}+\Gamma_{w_0}\\
&\geq&\Big(|A|-r+1\Big)x_{u^{*}}+r\bigg(1-\frac{r+1}{\rho}\bigg)x_{u^{*}}\\
&=&\bigg(|A|+1-\frac{r(r+1)}{\rho}\bigg)x_{u^{*}}\\
&>&|A|x_{u^{*}},
\end{eqnarray*}
which contradicts Lemma \ref{cla10}. The last inequality follows from (\ref{eq1}).

\vspace{1.8mm}
\noindent{\bf Case 1.2.} $b\geq2$.
\vspace{1.8mm}

Let $w_0$ be the center vertex and $w_1,\ldots,w_b$ be the leaf vertices of $K_{1,b}$. Similar to (\ref{eq8}), for $1\leq i\leq b$, we have
\begin{eqnarray}\label{eq9}
d_{A}(w_i)\leq|A|-1.
\end{eqnarray}
Now we claim that $d_{B}(w_0)=b\leq r$. Assume to the contrary that $b\geq r+1$. For $w_0w_1\in E(B)$, it follows from (\ref{eq6}) and (\ref{eq9}) that
\begin{eqnarray*}
\sum_{w\in B}\overline{d}_{A}(w)x_{u^{*}}&\geq&\bigg(\overline{d}_{A}(w_0)+\overline{d}_{A}(w_1)+\sum_{i=2}^{b}\overline{d}_{A}(w_i)\bigg)x_{u^{*}}\\
&\geq&\Big(|A|-r+b-1\Big)x_{u^{*}}\\
&\geq&|A|x_{u^{*}},
\end{eqnarray*}
contradicting Lemma \ref{cla10}. Hence $2\leq b\leq r$.

Let $d_{A}(w_0)=a$. Then we have $a\geq r$. In fact, if $a\leq r-1$, then we can construct a new graph $G'$ obtained from $G^{*}$ by adding $r-a$ edges between $w_0$ and $A\backslash N_{A}(w_0)$. Moreover, $G'$ is a non-bipartite $B_{r+1}$-free graph with $\rho(G')>\rho(G^{*})$, a contradiction. Now assume that $a=r+a_1$, where $a_1\geq0$. Next we shall divide the proof into the following two subcases according to different values of $a_1$.

\vspace{1.8mm}
\noindent{\bf Case 1.2.1.} $a_1\geq1$.
\vspace{1.8mm}

Since $G^{*}$ is $B_{r+1}$-free, we have $\overline{d}_{A}(w_i)\geq a_1$ for $1\leq i\leq b$. Hence $\sum_{i=1}^{b}\overline{d}_{A}(w_i)\geq ba_1$. Note that $\overline{d}_{A}(w_0)=|A|-r-a_1$ and $b\geq2$.
If $a_1\geq r$, then we have
\begin{eqnarray*}
\sum_{w\in B}\overline{d}_{A}(w)x_{u^{*}}&\geq&\bigg(\overline{d}_{A}(w_0)+\sum_{i=1}^{b}\overline{d}_{A}(w_i)\bigg)x_{u^{*}}\\
&\geq&\Big(|A|+(b-1)a_1-r\Big)x_{u^{*}}\\
&\geq&\Big(|A|+a_1-r\Big)x_{u^{*}}\\
&\geq&|A|x_{u^{*}},
\end{eqnarray*}
which contradicts Lemma \ref{cla10}. Hence $1\leq a_1\leq r-1$. Using the eigen-equation, we obtain that
$\rho x_{w_0}=\sum_{u\in N_{A}(w_0)}x_{u}+\sum_{i=1}^{b}x_{w_i}\leq(a+b)x_{u^{*}}=(r+a_1+b)x_{u^{*}}$. Then we have
\begin{eqnarray}\label{eq10}
x_{w_0}\leq\frac{r+a_1+b}{\rho}x_{u^{*}},
\end{eqnarray}
and hence $\Gamma_{w_0}=d_{A}(w_0)(x_{u^{*}}-x_{w_0})=(r+a_1)(x_{u^{*}}-x_{w_0})\geq(r+a_1)\Big(1-\frac{r+a_1+b}{\rho}\Big)x_{u^{*}}$. Then
\begin{eqnarray}\label{eq11}
\sum_{w\in B}\overline{d}_{A}(w)x_{u^{*}}+\sum_{w\in B}\Gamma_{w}&\geq&\Bigg(\overline{d}_{A}(w_0)+\sum_{i=1}^{b}\overline{d}_{A}(w_i)\Bigg)x_{u^{*}}+\Gamma_{w_0}\nonumber\\
&\geq&\Big(|A|-r-a_1+ba_1\Big)x_{u^{*}}+(r+a_1)\Bigg(1-\frac{r+a_1+b}{\rho}\Bigg)x_{u^{*}}\nonumber\\
&=&\Bigg(|A|+ba_1-\frac{(r+a_1)(r+a_1+b)}{\rho}\Bigg)x_{u^{*}}.
\end{eqnarray}
Let $g(b)=ba_1-\frac{(r+a_1)(r+a_1+b)}{\rho}$. Then $g'(b)=a_1-\frac{r+a_1}{\rho}$. By $1\leq a_1\leq r-1$ and (\ref{eq1}), we have
$g'(b)\geq1-\frac{2r-1}{\rho}>0$, which implies that $g(b)$ increases monotonically with respect to $b\in[2, r]$. By (\ref{eq11}), we have
\begin{eqnarray*}
\sum_{w\in B}\overline{d}_{A}(w)x_{u^{*}}+\sum_{w\in B}\Gamma_{w}&\geq&\Bigg(|A|+2a_1-\frac{(r+a_1)(r+a_1+2)}{\rho}\Bigg)x_{u^{*}}\\
&\geq&\Bigg(|A|+2-\frac{(2r-1)(2r+1)}{\rho}\Bigg)x_{u^{*}}\\
&>&|A|x_{u^{*}},
\end{eqnarray*}
contradicting Lemma \ref{cla10}.

\vspace{1.8mm}
\noindent{\bf Case 1.2.2.} $a_1=0$.
\vspace{1.8mm}

In this case, we have $d_{A}(w_0)=a=r$. It follows from (\ref{eq9}) that $\overline{d}_{A}(w_i)\geq1$ for $1\leq i\leq b$. Substitute $a_1=0$ into (\ref{eq10}), we have $x_{w_0}\leq\frac{r+b}{\rho}x_{u^{*}}$. Then
\begin{eqnarray*}
\sum_{w\in B}\overline{d}_{A}(w)x_{u^{*}}+\sum_{w\in B}\Gamma_{w}&\geq&\bigg(\overline{d}_{A}(w_0)+\sum_{i=1}^{b}\overline{d}_{A}(w_i)\bigg)x_{u^{*}}+\Gamma_{w_0}\\
&\geq&(|A|-r+b)x_{u^{*}}+r\bigg(1-\frac{r+b}{\rho}\bigg)x_{u^{*}}\\
&=&\bigg(|A|+b-\frac{r(r+b)}{\rho}\bigg)x_{u^{*}}.
\end{eqnarray*}
Recall that $2\leq b\leq r$. Combining (\ref{eq1}), we have
$$\sum_{w\in B}\overline{d}_{A}(w)x_{u^{*}}+\sum_{w\in B}\Gamma_{w}\geq\bigg(|A|+2-\frac{2r^{2}}{\rho}\bigg)x_{u^{*}}>|A|x_{u^{*}},$$
which contradicts Lemma \ref{cla10}.

\vspace{1.8mm}
\noindent{\bf Case 2.} $G^{*}[B]$ is a triangle with possibly some isolated vertices.
\vspace{1.8mm}

Let $w_0w_1w_2w_0$ be the triangle in $G^{*}[B]$. Note that $d_{B}(w_0)=d_{B}(w_1)=d_{B}(w_2)=2$. Without loss of generality, we assume that $d_{A}(w_0)=a$.
Then we have $a\geq r-1$. Otherwise we can obtain a non-bipartite $B_{r+1}$-free graph $G'$ with $\rho(G')>\rho(G^{*})$ by adding $r-1-a$ edges between $w_0$ and $A\backslash N_{A}(w_0)$ in $G^{*}$, a contradiction. Now assume that $a=r-1+a_1$, where $a_1\geq0$.

\vspace{1.8mm}
\noindent{\bf Case 2.1.} $a_1\geq1$.
\vspace{1.8mm}

Note that $G^{*}$ is $B_{r+1}$-free. Then $\overline{d}_{A}(w_i)\geq a_1$ for $1\leq i\leq 2$. Hence $\sum_{i=1}^{2}\overline{d}_{A}(w_i)\geq 2a_1$. Assume that $a_1\geq r-1$. Combining $\overline{d}_{A}(w_0)=|A|+1-r-a_1$, we have
\begin{eqnarray*}
\sum_{w\in B}\overline{d}_{A}(w)x_{u^{*}}&\geq&\Big(\overline{d}_{A}(w_0)+\overline{d}_{A}(w_1)+\overline{d}_{A}(w_2)\Big)x_{u^{*}}\\
&\geq&\Big(|A|+a_1+1-r\Big)x_{u^{*}}\\
&\geq&|A|x_{u^{*}},
\end{eqnarray*}
which contradicts Lemma \ref{cla10}. Hence $1\leq a_1\leq r-2$. Note that
$\rho x_{w_0}=\sum_{u\in N_{A}(w_0)}x_{u}+x_{w_1}+x_{w_2}\leq(a+2)x_{u^{*}}=(r+a_1+1)x_{u^{*}}$. Then we have
\begin{eqnarray}\label{eq10.0}
x_{w_0}\leq\frac{r+a_1+1}{\rho}x_{u^{*}},
\end{eqnarray}
and hence $\Gamma_{w_0}=d_{A}(w_0)(x_{u^{*}}-x_{w_0})=(r-1+a_1)(x_{u^{*}}-x_{w_0})\geq(r-1+a_1)\Big(1-\frac{r+a_1+1}{\rho}\Big)x_{u^{*}}$.
By $1\leq a_1\leq r-2$ and (\ref{eq1}), we have
\begin{eqnarray*}\label{eq11}
\sum_{w\in B}\overline{d}_{A}(w)x_{u^{*}}+\sum_{w\in B}\Gamma_{w}&\geq&\Big(\overline{d}_{A}(w_0)+\overline{d}_{A}(w_1)+\overline{d}_{A}(w_2)\Big)x_{u^{*}}+\Gamma_{w_0}\nonumber\\
&\geq&\Big(|A|-r+1+a_1\Big)x_{u^{*}}+(r-1+a_1)\Bigg(1-\frac{r+a_1+1}{\rho}\Bigg)x_{u^{*}}\nonumber\\
&=&\Bigg(|A|+2a_1-\frac{(r-1+a_1)(r+a_1+1)}{\rho}\Bigg)x_{u^{*}}\\
&\geq&\Bigg(|A|+2-\frac{(2r-3)(2r-1)}{\rho}\Bigg)x_{u^{*}}\\
&>&|A|x_{u^{*}},
\end{eqnarray*}
which contradicts Lemma \ref{cla10}.

\vspace{1.8mm}
\noindent{\bf Case 2.2.} $a_1=0$.
\vspace{1.8mm}

Now $d_{A}(w_0)=a=r-1$. For $1\leq i\leq 2$, similar to (\ref{eq8}), we have $d_{A}(w_i)\leq|A|-1$, which implies that $\overline{d}_{A}(w_i)\geq1$.
Substitute $a_1=0$ into (\ref{eq10.0}), we have $x_{w_0}\leq\frac{r+1}{\rho}x_{u^{*}}$. Combining (\ref{eq1}), we can obtain that
\begin{eqnarray*}
\sum_{w\in B}\overline{d}_{A}(w)x_{u^{*}}+\sum_{w\in B}\Gamma_{w}&\geq&\bigg(\overline{d}_{A}(w_0)+\overline{d}_{A}(w_1)+\overline{d}_{A}(w_2)\bigg)x_{u^{*}}+\Gamma_{w_0}\\
&\geq&(|A|-r+1)x_{u^{*}}+(r-1)\bigg(1-\frac{r+1}{\rho}\bigg)x_{u^{*}}\\
&=&\bigg(|A|+1-\frac{(r-1)(r+1)}{\rho}\bigg)x_{u^{*}}\\
&>&|A|x_{u^{*}},
\end{eqnarray*}
which also contradicts Lemma \ref{cla10}.
\end{proof}

To characterize the structure of $G^{*}$, we finally need to prove the following lemma.
\begin{lemma}\label{lem6}
If $G^{*}\cong K_{|A|-1,|B|+1}^{r,r}$, then $G^{*}\cong K_{\lfloor\frac{n-1}{2}\rfloor,\lceil\frac{n-1}{2}\rceil}^{r, r}$.
\end{lemma}
\begin{proof}
For convenience, we write $s=|A|-1$ and $t=|B|+1$. Then $s+t=|A|+|B|=n-1$. Without loss of generality, assume that $s\geq t$. It suffices to prove $s\leq t+1$.
Assume that $s\geq t+2$. Now we construct a new graph $G'=K_{s-1,t+1}^{r,r}$.
According to Lemma \ref{lem4}, we know that $\rho\big(K_{s-1,t+1}^{r,r}\big)$ is the largest root of $f(s-1,t+1,r,x)=0$. Moreover,
$$f(s-1,t+1,r,x)-f(s,t,r,x)=x\cdot\Big((t+1-s)x^{2}+2(s-t-1)r\Big)<0$$
for $x>\sqrt{2r}$. By (\ref{eq1}), we have $\rho\big(K_{s,t}^{r,r}\big)>4(r^{2}+r+1)>\sqrt{2r}$. Hence $\rho\big(K_{s-1,t+1}^{r,r}\big)>\rho\big(K_{s,t}^{r,r}\big)$.
Note that $K_{s-1,t+1}^{r,r}$ is still a non-bipartite $B_{r+1}$-free graph. This contradicts the maximality of $\rho(G^{*}).$
\end{proof}

Now we are in a position to present the proof of Theorem \ref{main}.

\medskip
\noindent  \textbf{Proof of Theorem \ref{main}}.
By Theorem \ref{cla11}, $e(A)\neq0$, and hence $\nu(G^{*}[A])=1$. By Lemma \ref{cla9}, we have $\nu(G^{*}[A])+\nu(G^{*}[B])=1$, and hence $B$ is an independent set. Now we divide the proof into two cases according to the structure of $G^{*}[A]$.

\vspace{1.8mm}
\noindent{\bf Case 1.} $G^{*}[A]$ is a star $K_{1,a}$ with possibly some isolated vertices.
\vspace{1.8mm}

In this case, we divide the proof into the following two cases according to the different values of $a$.

\vspace{1.8mm}
\noindent{\bf Case 1.1.} $a=1$.
\vspace{1.8mm}

Let $u_0u_1$ be the only edge in $E(A)$. By the maximality of $\rho(G^{*})$, we have $N_{B}(u_0)\neq\varnothing$ and $N_{B}(u_1)\neq\varnothing$.
Without loss of generality, assume that $x_{u_0}\leq x_{u_1}$. Now we have the following claim.

\begin{claim}\label{clac}
$N_{B}(u_0)\subseteq N_{B}(u_1)$.
\end{claim}
\begin{proof}
If $N_{B}(u_0)\nsubseteq N_{B}(u_1)$, then there exists a vertex $w\in N_{B}(u_0)\backslash N_{B}(u_1)$. Let $G'=G^{*}-wu_0+wu_1$. Then $G'$ is still $B_{r+1}$-free. If not, $G'$ contains a $B_{r+1}$ and $wu_1\in E(B_{r+1})$, which contradicts that $N_{G'}(u_1)\cap N_{G'}(w)=\varnothing$. Note that $G'$ is still non-bipartite.
By Lemma \ref{lem2}, $\rho(G')>\rho(G^{*})$, which contradicts the maximality of $\rho(G^{*})$.
\end{proof}
\begin{figure}
\centering
\includegraphics[width=0.6\textwidth]{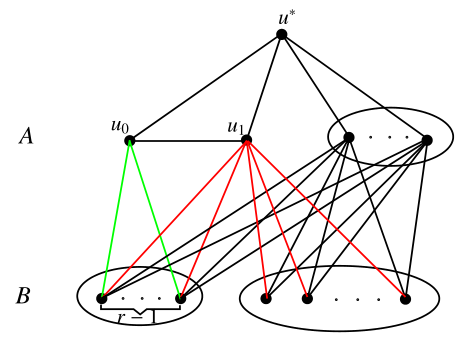}\\
\caption{The graph $G^{*}$ with $e(A)=1$.}
\label{f5}
\end{figure}
\begin{claim}\label{clad}
$d_{B}(u_0)=r-1.$
\end{claim}
\begin{proof}
By Claim \ref{clac}, $N_{B}(u_0)\subseteq N_{B}(u_1)$. Note that $G^{*}$ is $B_{r+1}$-free. Then we have $d_{B}(u_0)\leq r-1$. If $d_{B}(u_0)\leq r-2$, then we can construct a new graph $G'$ obtained from $G^{*}$ by adding $r-1-d_{B}(u_0)$ edges between $u_0$ and $N_{B}(u_1)\backslash N_{B}(u_0)$. Note that $G'$ is a non-bipartite $B_{r+1}$-free graph with $\rho(G')>\rho(G^{*})$. This contradicts the maximality of $\rho(G^{*})$. Hence we have $d_{B}(u_0)=r-1$.
\end{proof}

By the maximality of $\rho(G^{*})$ again, then $N_{B}(u)=B$ for any vertex $u\in A\backslash\{u_0\}$. Now we determine the structure of $G^{*}$ as Fig. \ref{f5}.

Next we shall prove that $r=1$. Assume that $r\geq2$. By Lemma \ref{cla2}, $d_{A}(u_0)\leq r$. Then we can construct a new graph $G'$ by adding $r-1$ edges between $u_0$ and $A\backslash\{u_0, u_1\}$. However, $G'$ is a still non-bipartite $B_{r+1}$-free graph
and $\rho(G')>\rho(G^{*})$, a contradiction. Hence $r=1$. By Claim \ref{clad}, $d_{B}(u_0)=0.$
Now we obtain that $G^{*}-\{u_0\}\cong K_{|A|-1,|B|+1}$ and $d_{A\backslash\{u_0\}}(u_0)=d_{B\cup\{u^{*}\}}(u_0)=1$. Hence $G^{*}\cong K_{|A|-1,|B|+1}^{1,1}$. By Lemma \ref{lem6}, we have $G^{*}\cong K_{\lfloor\frac{n-1}{2}\rfloor,\lceil\frac{n-1}{2}\rceil}^{1,1}$.

\vspace{1.8mm}
\noindent{\bf Case 1.2.} $a\geq2$.
\vspace{1.8mm}

In this case, $G^{*}[A]$ consists of a star $K_{1,a}$ and possibly some isolated vertices, where $e(A)=a\geq2$. Let $u_0$ be the center vertex and $u_1,\ldots,u_a$ be leaf vertices of $K_{1,a}$. By Lemma \ref{cla2}, we have $a\leq r$. Hence $r\geq2$. Let $d_{B}(u_0)=b$.

\begin{claim}\label{clae}
$b\geq r-1$.
\end{claim}
\begin{proof}
Assume that $b\leq r-2$. We construct a new graph $G'$ obtained from $G^{*}$ by adding $r-1-d_{B}(u_0)$ edges between $u_0$ and $B\backslash N_{B}(u_0)$.
Then $G'$ is still a non-bipartite $B_{r+1}$-free graph with $\rho(G')>\rho(G^{*})$, which contradicts the maximality of $\rho(G^{*})$. Hence $b\geq r-1$.
\end{proof}

By Claim \ref{clae}, we assume that $b=r-1+b_1$, where $b_1\geq0$.
Let $B_0=\{w_{1},\ldots,w_{r-1}\}$ be the $r-1$ vertices in $N_{B}(u_0)$ with largest eigenvector entry, i.e., such that $x_{w_j}\geq x_{w}$, where $1\leq j\leq r-1$ and $w\in N_{B}(u_0)\backslash B_0$.

\begin{claim}\label{claf}
$N_{N_{B}(u_0)}(u_i)=B_0$ for $1\leq i\leq a.$
\end{claim}
\begin{proof}
Since $G^{*}$ is $B_{r+1}$-free, we have $d_{N_{B}(u_0)}(u_i)\leq r-1$ for $1\leq i\leq a$. Furthermore, by the maximality of $\rho(G^{*})$, we have $d_{N_{B}(u_0)}(u_i)=r-1$ for $1\leq i\leq a$.
Now we prove that $N_{N_{B}(u_0)}(u_i)=B_0$ for $1\leq i\leq a$. If $N_{N_{B}(u_0)}(u_i)\neq B_0$ for some $i$, then there exists a vertex $w\in N_{B}(u_0)\backslash B_0$ such that $u_iw\in E(G^{*})$. Since $d_{N_{B}(u_0)}(u_i)=r-1$, there must exist some vertex $w_j$ such that $u_iw_j\notin E(G^{*})$, where $1\leq j\leq r-1$. Let $G'=G^{*}-u_iw+u_iw_j$.
We claim that $G'$ is still $B_{r+1}$-free. In fact, if $G'$ contains a $B_{r+1}$, then $u_iw_j\in E(B_{r+1}).$ Note that $N_{G'}(u_i)\cap N_{G'}(w_j)=\{u_0\}.$ Then $|N_{G'}(u_i)\cap N_{G'}(u_0)|\geq r+1$, which contradicts that $|N_{G'}(u_i)\cap N_{G'}(u_0)|=r$. Note that $G'$ is still non-bipartite. By Lemma \ref{lem2}, we have $\rho(G')>\rho(G^{*})$, a contradiction.
\end{proof}

For convenience, we write $B_1=N_{B}(u_0)\backslash B_0$ and $B_2=B\backslash(B_0\cup B_1)$. Then $|B_1|=b_1$. Let $|B_2|=b_2$. By the maximality of $\rho(G^{*})$ again, we can obtain that $N_{B}(u_i)=B_0\cup B_2$ for $1\leq i\leq a$ and $N_{B}(u)=B$ for any vertex $u\in A\backslash\{u_0,u_1,\ldots,u_a\}$. Combining Claim \ref{claf}, we determine the structure of $G^{*}$ as Fig. \ref{f6}.

\begin{claim}\label{clag}
$b_1=0$.
\end{claim}
\begin{proof}
Suppose to the contrary that $b_1\geq1$. Then $\overline{d}_{B}(u_i)=b_1$ for $1\leq i\leq a$. Note that $\overline{d}_{B}(u_0)=|B|-r+1-b_1$. Then
\begin{eqnarray*}
\rho^{2}x_{u^{*}}&=&|A|x_{u^{*}}+\sum_{u\in A}d_{A}(u)x_{u}+\sum_{w\in B}d_{A}(w)x_{w}\\
&\leq&|A|x_{u^{*}}+2e(A)x_{u^{*}}+e(A, B)x_{u^{*}}\\
&=&|A|x_{u^{*}}+2ax_{u^{*}}+\big(|A||B|-|B|+r-1+b_1-ab_1\big)x_{u^{*}}\\
&=&\Big((|A|-1)(|B|+1)+2a+r-(a-1)b_1\Big)x_{u^{*}}.
\end{eqnarray*}
\begin{figure}
\centering
\includegraphics[width=0.65\textwidth]{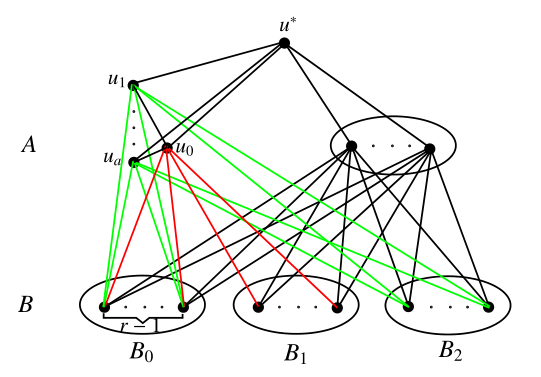}
\caption{The graph $G^{*}$ with $e(A)\geq2$.}
\label{f6}
\end{figure}
Recall that $|A|+|B|=n-1$ and $a\geq2$. Combining (\ref{eq2}), we have
$$\bigg\lfloor\frac{(n-1)^{2}}{4}\bigg\rfloor<\rho^{2}\leq\bigg\lfloor\frac{(n-1)^{2}}{4}\bigg\rfloor+2a+r-(a-1)b_1.$$
It follows that $b_1<\frac{2a+r}{a-1}=2+\frac{r+2}{a-1}\leq r+4$, that is, $b_1\leq r+3$.

By symmetry, we have $x_{u_i}=x_{u_1}$ for $2\leq i\leq a$. Note that $x_w=x_{u^{*}}$ for any vertex $w\in B_0$.
For each vertex $w\in B_2$, we have $\rho x_w=\rho x_{u^{*}}-x_{u_0}$. Recall that $|B_0|=r-1$.
Note that $|A|+|B|+1=n$. By Lemma \ref{cla1}, we have $|B|+1\leq\big\lfloor\frac{n}{2}\big\rfloor$. Note that $b_1\geq1$. Combining (\ref{eq1}), we can obtain that $b_2\leq|B|-1\leq\big\lfloor\frac{n}{2}\big\rfloor-2\leq\frac{n-4}{2}<\rho$. Then
\begin{eqnarray}\label{eq12}
\rho x_{u_1}&=&x_{u^{*}}+x_{u_0}+\sum_{w\in B_0}x_{w}+\sum_{w\in B_2}x_{w}\nonumber\\
&=&rx_{u^{*}}+x_{u_0}+b_2\bigg(x_{u^{*}}-\frac{x_{u_0}}{\rho}\bigg)\nonumber\\
&=&(r+b_2)x_{u^{*}}+\bigg(1-\frac{b_2}{\rho}\bigg)x_{u_0}\nonumber\\
&>&(r+b_2)x_{u^{*}}.
\end{eqnarray}
Moreover, we have
\begin{eqnarray}\label{eq13}
\rho x_{u_0}=x_{u^{*}}+ax_{u_1}+\sum_{w\in B_0\cup B_1}x_{w}\leq\big(r+a+b_1\big)x_{u^{*}}.
\end{eqnarray}
Recall that $a\geq2$ and $b_1\leq r+3$.
Therefore, combining (\ref{eq12}) and (\ref{eq13}), we obtain that
\begin{eqnarray*}
\rho(ax_{u_1}-x_{u_0})&>&a(r+b_2)x_{u^{*}}-(r+a+b_1)x_{u^{*}}=\Big((a-1)r+(b_2-1)a-b_1\Big)x_{u^{*}}\\
&\geq&\Big((r+2(b_2-1)-b_1\Big)x_{u^{*}}=\Big(r+2(|B|-r-b_1)-b_1\Big)x_{u^{*}}\\
&=&\big(2|B|-r-3b_1\big)x_{u^{*}}\geq\big(2|B|-4r-9\big)x_{u^{*}}\\
&=&\Big(2(|B|+1)-(4r+11)\Big)x_{u^{*}}.
\end{eqnarray*}
By Lemma \ref{cla3}, $|B|+1\geq5r+3$. Note that $r\geq2$. Hence we have
\begin{eqnarray}\label{eq14}
\rho(ax_{u_1}-x_{u_0})>(6r-5)x_{u^{*}}>0.
\end{eqnarray}
Let $B_1=\{w_j|1\leq j\leq b_1\}$.
Now we construct a new graph $G'=G^{*}-\{u_0w_j|1\leq j\leq b_1\}+\{u_iw_j|1\leq i\leq a, 1\leq j\leq b_1\}$.
Then $G'\cong K_{|A|-1,|B|+1}^{a,r}$ and $K_{|A|-1,|B|+1}^{a,r}$ is still a non-bipartite $B_{r+1}$-free graph. By symmetry, we have $x_{w_1}=x_{w_2}=\cdots=x_{w_{b_1}}$. Combining (\ref{eq14}), we have
\begin{eqnarray*}
\rho(G')-\rho(G^{*})&\geq& \textbf{x}^{T}A(G')\textbf{x}-\textbf{x}^{T}A(G^{*})\textbf{x}\\
&=&\sum\limits_{1\leq i\leq a, 1\leq j\leq b_1}2x_{w_j}x_{u_i}-\sum\limits_{1\leq j\leq b_1}2x_{w_j}x_{u_0}\\
&=&2ab_1x_{w_1}x_{u_1}-2b_1x_{w_1}x_{u_0}\\
&=&2b_1x_{w_1}\big(ax_{u_1}-x_{u_0}\big)\\
&>&0,
\end{eqnarray*}
which contradicts the maximality of $\rho(G^{*})$.
Hence we have $b_1=0$.
\end{proof}

By Claim \ref{clag}, then $d_{B}(u_0)=b=r-1$. Furthermore, by the maximality of $\rho(G^{*})$, we have $d_{A}(u_0)=a=r$.
Now we have $G^{*}-\{u_0\}\cong K_{|A|-1,|B|+1}$ and $d_{A\backslash\{u_0\}}(u_0)=d_{B\cup\{u^{*}\}}(u_0)=r$.
This implies that $G^{*}\cong K_{|A|-1,|B|+1}^{r,r}$. By Lemma \ref{lem6}, we have $G^{*}\cong K_{\lfloor\frac{n-1}{2}\rfloor,\lceil\frac{n-1}{2}\rceil}^{r,r}$.

\vspace{1.8mm}
\noindent{\bf Case 2.} $G^{*}[A]$ is a triangle with possibly some isolated vertices.
\vspace{1.8mm}

Let $u_0u_1u_2u_0$ be the triangle in $G^{*}[A]$. Note that $d_{A}(u_0)=d_{A}(u_1)=d_{A}(u_2)=2$. Without loss of generality, we assume that $d_{B}(u_0)=b$. Then we have $b\geq r-2$. Otherwise, we can obtain a non-bipartite $B_{r+1}$-free graph $G'$ with $\rho(G')>\rho(G^{*})$ by adding $r-2-b$ edges between $u_0$ and $B\backslash N_{B}(u_0)$, a contradiction. Let $b=r-2+b_1$, where $b_1\geq0$. Note that $e(A)=3$. By using the eigen-equation, we have
\begin{eqnarray}\label{eq10.1}
\rho^{2}x_{u^{*}}&=&|A|x_{u^{*}}+\sum_{u\in A}d_{A}(u)x_{u}+\sum_{w\in B}d_{A}(w)x_{w}\nonumber\\
&\leq&|A|x_{u^{*}}+2e(A)x_{u^{*}}+e(A, B)x_{u^{*}}\nonumber\\
&=&|A|x_{u^{*}}+6x_{u^{*}}+\bigg(|A||B|-\sum_{u\in A}\overline{d}_{B}(u)\bigg)x_{u^{*}}.
\end{eqnarray}
We first prove the following claim.
\begin{claim}\label{clac.0}
$\sum_{u\in A}\overline{d}_{B}(u)<|B|+7.$
\end{claim}
\begin{proof}
Suppose that $\sum_{u\in A}\overline{d}_{B}(u)\geq|B|+7.$ Recall that $|A|+|B|=n-1$. Combining (\ref{eq10.1}), we obtain that
\begin{eqnarray*}
\rho^{2}x_{u^{*}}&\leq&|A|x_{u^{*}}+6x_{u^{*}}+\Big(|A||B|-(|B|+7)\Big)x_{u^{*}}\\
&=&(|A|-1)(|B|+1)x_{u^{*}}\\
&\leq&\bigg\lfloor\frac{(n-1)^{2}}{4}\bigg\rfloor x_{u^{*}},
\end{eqnarray*}
which contradicts (\ref{eq2}). Hence $\sum_{u\in A}\overline{d}_{B}(u)<|B|+7.$
\end{proof}

Next we consider the following two cases according to different values of $b_1$.

\vspace{1.8mm}
\noindent{\bf Case 2.1.} $b_1\geq r+5$.
\vspace{1.8mm}

Since $G^{*}$ is $B_{r+1}$-free, we have $\overline{d}_{B}(u_i)\geq b_1$ for $1\leq i\leq 2$. Hence $\sum_{i=1}^{2}\overline{d}_{B}(u_i)\geq 2b_1$. Note that $\overline{d}_{B}(u_0)=|B|+2-r-b_1$. Then
\begin{eqnarray*}
\sum_{u\in A}\overline{d}_{B}(u)&\geq&\sum_{i=0}^{2}\overline{d}_{B}(u_i)
\geq|B|+2+b_1-r
\geq|B|+7,
\end{eqnarray*}
which contradicts Claim \ref{clac.0}.

\vspace{1.8mm}
\noindent{\bf Case 2.2.} $0\leq b_1\leq r+4$.
\vspace{1.8mm}

Note that $d_{B}(u_0)=r-2+b_1\leq 2r+2$. Then $\overline{d}_{B}(u_0)\geq|B|-2r-2$. Let $B_0=N_{B}(u_0)$.  By the maximality of $\rho(G^{\ast})$, we have $B\backslash B_0\subseteq N_{B}(u_1)\cup N_{B}(u_2)$. Denote by $B_1=N_{B\backslash B_0}(u_1)\cap N_{B\backslash B_0}(u_2)$. Since $G^{\ast}$ is $B_{r+1}$-free, $|B_1|\leq r-2$. Without loss of generality, we assume that $x_{u_1}\geq x_{u_2}$. Then we have $B\backslash(B_0\cup B_1)\subseteq N_{B}(u_1)$. Otherwise, there exists a vertex $w\in B\backslash(B_0\cup B_1)$ such that $w\in N_{B}(u_2)\backslash N_{B}(u_1)$. Then $G'=G^{\ast}-wu_2+wu_1$ is a non-bipartite $B_{r+1}$-free graph with $\rho(G')>\rho(G^{*})$, a contradiction. Therefore, $B\backslash B_0\subseteq N_{B}(u_1)$. Note that $G^{\ast}$ is $B_{r+1}$-free. Then $|N_{B_0}(u_2)|\leq r-2$ and $|N_{B\backslash B_0}(u_2)|\leq r-2$, and hence $\overline{d}_{B}(u_2)\geq|B|-2r+4$. Combining Lemma \ref{cla3}, we have
\begin{eqnarray*}
\sum_{u\in A}\overline{d}_{B}(u)&\geq&\overline{d}_{B}(u_0)+\overline{d}_{B}(u_2)
\geq2|B|-4r+2\\
&=&|B|+(|B|+1-(4r-1))
\geq|B|+9,
\end{eqnarray*}
which contradicts Claim \ref{clac.0}.
\hspace*{\fill}$\Box$

\section{Concluding remarks}\label{se4}
In this section, we focus on the Brualdi-Hoffman-Tur\'{a}n Problem on books, i.e., determining the maximum spectral radius of $B_{r+1}$-free graphs on $m$ edges. Since isolated vertices do not have an effect on the spectral radius, we here consider graphs without
isolated vertices. In 2021, Zhai, Lin and Shu \cite{Zhai2021} posed the following conjecture on books.
\begin{conj}
Let $G$ be a $B_{r+1}$-free graph on $m$ edges and $m\geq m_{0}(r)$ for large enough $m_{0}(r)$. Then
$$\rho(G)\leq\sqrt{m},$$
with equality if and only if $G$ is a complete bipartite graph.
\end{conj}
Let $bk(G)$ stand for the booksize of $G$, that is, the maximum number of triangles with a common edge in $G$.
Recently, Nikiforov \cite{Nikiforov2021} confirmed the above conjecture.
\begin{theorem}
If $G$ is a graph with $m$ edges and $\rho(G)\geq\sqrt{m}$, then
$$bk(G)>\frac{1}{12}\sqrt[4]{m},$$
unless $G$ is a complete bipartite graph.
\end{theorem}
Note that complete bipartite graphs attain the maximum spectral radius among all $B_{r+1}$-free graph with size $m$. Hence Nikiforov's theorem inspire us to investigate the maximum spectral radius of non-bipartite $B_{r+1}$-free graph with size $m$ and characterize the corresponding spectral extremal graphs.
Denote by $S_{m,1}^{+}$ the graph obtained from $K_{1,m-1}$ by adding an edge within its independent set. Now we propose the following conjecture.
\begin{conj}
Let $G$ be a non-bipartite $B_{r+1}$-free graph with size $m$. Then
$$\rho(G)\leq\rho\big(S_{m,1}^{+}\big),$$
with equality if and only if $G\cong S_{m,1}^{+}$.
\end{conj}

\section*{Declaration of competing interest}

The authors declare that they have no conflict of interest.

\section*{Data availability}

No data was used for the research described in the article.

\section*{Acknowledgements}

The authors would like to thank the anonymous referees for their helpful comments on improving the presentation of the paper. The research of Ruifang Liu is supported by the National Natural Science Foundation of China (Nos. 12371361 and 12171440), Distinguished Youth Foundation of Henan Province (No. 242300421045).

\end{document}